\documentclass{jhrs}

\usepackage{amsxtra,amssymb,amsfonts,amsthm,amsmath,latexsym,stmaryrd}
\usepackage{diagrams}

\theoremstyle{plain}
\newtheorem{definition}{Definition}
\newtheorem{theorem}{Theorem}

\newtheorem*{theorem*}{Theorem}
\newtheorem{lemma}{Lemma}
\newtheorem{corollary}{Corollary}
\newtheorem{example}{Example}

\newcommand{\refT}[1]{Theorem~\ref{T:#1}}

\newcommand{\refL}[1]{Lemma~\ref{L:#1}}

\newcommand{\refE}[1]{Example~\ref{E:#1}}

\newcommand{\refS}[1]{Section~\ref{S#1}}

\def\lra{\longrightarrow}
\def\hra{\hookrightarrow}
\def\xra{\xrightarrow}

\def\lla{\longleftarrow}
\def\ovs{\overset}

\def\pr{\mathop{\rm pr}\nolimits}
\def\R{{\mathbb R}}

\def\Z{{\mathbb Z}}

\def\H{{\mathbb H}}

\def\O{{\mathcal O}}

\def\g{{\mathfrak g}}

\def\g{{\mathfrak g}}

\def\i{\imath}
\def\pfi{\varphi}

\def\oH{\buildrel\circ\over H}
\def\oH1{\buildrel\circ\over H\kern-.02in{}^1}
\def\hookuparrow{{\cup\kern-.04in{}^\uparrow}}

\def\b{\mathbf b}

\def\Im{\mathop{\rm Im}\nolimits}

\def\const{\mathop{\rm const}\nolimits}

\def\Stab{\mathop{\rm Stab}\nolimits}

\def\Ext{\mathop{\rm Ext}\nolimits}
\def\Hom{\mathop{\rm Hom}\nolimits}
\def\pr{\mathop{\rm pr}\nolimits}

\def\id{\mathop{\rm id}\nolimits}
\def\pt{\mathop{\rm pt}\nolimits}

\def\tr{\mathop{\rm tr}\nolimits}

\def\bee{\begin{equation}}
\def\eee{\end{equation}}
\def\be{\begin{equation*}}
\def\ee{\end{equation*}}
\def\bal{\begin{aligned}}
\def\eal{\end{aligned}}

\begin{document}
\title{Homotopy classification\\ of maps into homogeneous spaces}          

\author{Sergiy Koshkin}             
\email{koshkinS@uhd.edu}       
\address{Department of Computer and Mathematical Sciences\\ 
        University of Houston-Downtown\\
        1 Main Street \#S705\\ 
	Houston, TX 77002}



\begin{abstract}
We give an alternative to Postnikov's homotopy classification of
maps from 3-dimensional CW-complexes to homogeneous spaces $G/H$ of Lie groups. 
It describes homotopy classes in terms of lifts to the group $G$ and is suitable 
for extending the notion of homotopy to Sobolev maps. This is required in applications 
to variational problems of mathematical physics.
\end{abstract}

\submitted{Jim Stasheff}  

\volumeyear{2009} 
\volumenumber{4}  
\issuenumber{1}   

\startpage{1}     

\maketitle

\section*{Introduction}

A classical theorem of Postnikov \cite{Ps,WJ} gives a homotopy classification of continuous maps from a $3$--dimensional CW--complex $M$ to a connected simply connected complex $X$ of any dimension. First it gives the primary invariant that characterizes when $M\ovs{\psi,\pfi}\lra X$ are $2$--homotopic, i.e. their restrictions to a $2$--skeleton of $M$ are homotopic. When this happens a secondary invariant is defined along with a condition that makes $\psi$ and $\pfi$ homotopic. Unfortunately, the secondary invariant is hard to compute because one has to homotop one of the maps into the other on the $2$--skeleton. While the primary invariant can be characterized in terms of deRham cohomology and thus extended to discontinuous (Sobolev) maps, the secondary one of Postnikov is tied too closely to restrictions and homotopy that do not make sense without continuity.

In physical applications one is often forced to extend topological notions to Sobolev maps. This requires rethinking characterizations of homotopy classes in a way that makes such extension possible. One way is to interpret homotopy classes as connected components in the space of continuous maps. A natural generalization is to study components in spaces of Sobolev maps \cite{HL,Wh}. However, in applications to mathematical physics a more hands on approach is usually taken. One identifies invariants that characterize homotopy classes and then extends them to Sobolev maps. This approach is taken in recent works on the Faddeev model \cite{AK2,LY} with maps into $S^2=SU(2)/U(1)$. The present work was motivated by considering its generalization, the Faddeev-Niemi model \cite{FN} with targets $SU( N)/T$, $T$ the maximal torus. In this paper we  give an alternative description only for homotopy classes of continuous maps. The extension to Sobolev maps and applications will be published elsewhere \cite{K}.

We consider continuous maps from $3$--dimensional CW--complexes to compact simply connected homogeneous spaces of Lie groups. In exchange for loss in generality one gets a much more explicit description of the secondary invariant and its deRham presentation. It sheds new light on the primary invariant as well. 

A smooth manifold is called a homogenous space under an action of a Lie group $G$ if the action is transitive. Any homogenous space $X$ can be identified with the coset space $G/H$, where $H$ is the isotropy subgroup of a point. Up to a diffeomorphism, different pairs $G,H$ may produce the same space $X$. Throughout this paper {\sl $X$ will always denote $G/H$ where $G$ is compact, connected and simply connected and $H\subset G$ is connected}. This can be done without loss of generality by switching to a maximal compact subgroup, universal cover and/or identity component of $G$ as appropriate \cite{BtD,Mg}. Under this assumtion, we prove that $\psi$ and $\pfi$ are $2$-homotopic  if and only if there exists a continuous lift $M\ovs{u}\lra G$ such that $\psi=u\pfi$ (\refT{1.1}), where $u\pfi$ refers to the action of $G$ on $X$. This easily generalizes if we allow $u$ to be a Sobolev map but the real advantage is that the secondary invariant for $\psi,\pfi$ becomes the primary one for $u$.

It is convenient to introduce the basic class of a space $F$. Suppose 
$\pi_0(F)=...=\pi_{n-1}(F)=0$ then by the Hurewicz theorem $H_n (F,\mathbb{Z})\simeq\pi_n(F)$. 
The {\sl basic class $\b_F\in H^n(F,\pi_n(F))$} is the cohomology class that maps every homology class in $H_n(F,\Z)$ 
into its image in $\pi_n(F)$ under the Hurewicz isomorphism ($\b_F$ is called the fundamental class of $F$ by Steenrod). It follows essentially from the Eilenberg classification theorem \cite{St} that $\psi^*\b_F$ is the primary invariant for homotopy. Namely, two maps $M\ovs{\psi,\pfi}\lra F$ are $n$-homotopic, i.e. their restrictions to the $n$-skeleton are homotopic iff $\psi^*\b_F=\pfi^*\b_F$. 

Let $\b_G\in H^3(F,\pi_3(G))$ be the basic class of $G$ then the secondary invariant is $u^*\b_G$. However, $\psi,\pfi$ being homotopic is not quite equivalent to its vanishing. The problem is that $u$ in $\psi=u\pfi$ is not unique. Potentially, there are maps $M\ovs{w}\lra G$ with $w\pfi=\pfi$ but $w^*\b_G\ne0$. One has to factor out the subgroup 
generated by such maps
$$
\O_\pfi:=\{w^*\b_G\mid w\pfi=\pfi\}\subset H^3(M,\pi_3(G)).
$$
Despite its appearence, this subgroup depends only on the $2$-homotopy class of $\pfi$ (\refL{1.5}). Our main results (Theorems \ref{T:1.1} and \ref{T:1.2}) can be summarized as follows.
\begin{theorem*}Let $X=G/H$ be a compact simply connected homogeneous space and $M$ a $3$-dimensional CW--complex. Then two maps $M\ovs{\psi,\pfi}\lra X$ are homotopic if and only if there exists a map $M\ovs{u}\lra G$ such that $\psi=u\pfi$ and $u^*\b_G\in\O_\pfi$. Within the $2$-homotopy class of $\pfi$, the homotopy classes are in one-to-one correspondence with $H^3(M,\pi_3(G))/\O_\pfi$.
\end{theorem*}
When $H^2(M,\Z)=0$ any two maps $M\ovs{\psi,\pfi}\lra X$ are $2$-homotopic and therefore related by a lift $u$. We can always choose $\pfi$ to be the constant map and define the secondary invariant for a single map $\psi$. The subgroup $\O_{\const}$ is trivial and homotopy classes are in one-to-one correspondence with $H^3(M,\pi_3(G))$. When $M=S^3$ and $X=S^2=SU_2/U_1$ one has $\pi_3(G)\simeq\Z$ and our $u^*\b_G$ is essentially the Hopf invariant, cf. \cite{LY}. The class $\b_G$ has a particularly nice deRham presentation when the group $G$ is simple. Then always $\pi_3(G)\simeq\Z$ and one can identify it with the class of the integral form $\Theta:=c_G\tr(g^{-1}dg\wedge g^{-1}dg \wedge g^{-1}dg )$. Here $c_G$ are normalizing constants computed in \cite{AK1}, for example $c_{SU_N}=-\frac1{96\pi^2N}$. The pullback is
$$
u^*\Theta=c_G\tr(u^{-1}du\wedge u^{-1}du \wedge u^{-1}du )
$$ 
and $\int_Mu^*\Theta$ can be made sense of for Sobolev $u$.

It is instructive to compare our secondary invariant with Postnikov's \cite{Ps,WJ}. His definition requires homotoping $\psi$ into a function $\widetilde{\psi}$ equal to $\pfi$ on the $2$-skeleton of $M$ and computing the primary difference $\overline{d}(\pfi,\widetilde{\psi})\in H^3(M,\pi_3(X))$. Homotopy occurs when this difference takes value in a subgroup with a complicated description that involves the Whitehead product \cite{WG} and the Postnikov square \cite{Nk1,Ps,WJ}. Note that by the homotopy exact sequence
\bee
0=\pi_2(H)\overset{\partial}\lla\pi_3(G/H)\overset{\pi_*}\lla\pi_3(G)\overset{i_*}\lla\pi_3(H)\lla\cdots
\eee
and $\pi_3(X)\simeq\pi_3(G)/\i_*\pi_3(H)$. If $\i_*\pi_3(H)=0$ as in the case of $U_1$ in $SU_2$ our invariant can be identified with Postnikov's. 

In \refS{1} we solve the relative lifting problem for two maps as a problem in  obstruction theory. A key role is played by the {\sl bundle of shifts} that helps characterize existence of the lift in terms of the {\sl primary characteristic class} of the quotient bundle $G\to G/H$. In \refS{2} we show that the primary characteristic class is essentially the basic class $\b_X$ and prove our characterization of $2$-homotopy classes. Finally, in \refS{3} the secondary invariant is introduced and the homotopy classification is completed. We also give a deRham interpretation of the secondary invariant.

\section{Primary characteristic class and lifting}\label{S1}

In this section we will define a cohomology class on $G/H$ that regulates existence of a relative lift $u$ such that $\psi=u\pfi$ for two maps $M\ovs{\psi,\pfi}\lra G/H$. This class is the primary characteristic class \cite{MS,St} of the bundle $G\to G/H$ denoted $\varkappa(G)$. Of course, $\varkappa (G)$ also depends on $H\subset G$ but we follow the usual abuse of notation. The lift exists iff pullbacks of $\varkappa(G)$ by both maps are the same (\refT{1.0}). We will prove this by reducing the lifting problem to a problem in the obstruction theory \cite{MS,St}. In the next section we will identify $\varkappa(G)$ with the primary obstruction to homotopy.

Here is the basic idea of the proof. Given two maps
$M\xra{\pfi ,\psi}X$ define $M\overset{(\pfi,\psi)}\lra X\times X$ 
and consider the {\it ratio bundle} over $M$:
\begin{equation}\label{e0.23}
\bal
Q_{\pfi ,\psi}:=\{(m,g)\in M\times G|\psi(m)=g\pfi(m)\}
\eal
\end{equation}
Obviously, sections of this bundle
$M\overset{\sigma} \lra Q_{\pfi ,\psi}\subset M\times G$ have the form $\sigma (m)=(m,u(m))$,
where $\psi =u\pfi$. In other words, they play the role of non-existent ratios $\psi/\pfi$.
Hence the problem of finding a lift $u$ is equivalent to constructing a section of the bundle $Q_{\pfi ,\psi}$,
which is a standard problem in the obstruction theory.

First, we have to establish that $Q_{\pfi ,\psi}$ are indeed fiber bundles. Note that 
\begin{equation}\label{e0.231}
\bal
Q_{\pfi ,\psi}\simeq\{(m,x,g)\in M \times X\times G|(\pfi (m),\psi(m)=(x,gx)\}
\eal
\end{equation}
and by definition of pullback $Q_{\pfi ,\psi}\simeq(\pfi ,\psi)^*Q$, where $Q$ is the bundle of shifts defined next. A particular case of this bundle was used in \cite{AS} for similar purposes.
\begin{definition}[The bundle of shifts]\label{D:shifts}
The bundle of shifts of a homogeneous space
$G/H=X$ is the fiber bundle $Q$ over $X\times X$ given by:
\begin{equation}\label{shifts}
\begin{aligned}
X\times G\overset{\alpha} \lra X\times X.\\ 
(x,g)\longmapsto (x,gx)
\end{aligned}
\end{equation}
\end{definition}
Thus, it suffices to prove that $Q$ itself is a fiber bundle. We will do more and identify the principal bundle it is associated to. It turns out to be the Cartesian double $G\times G \overset{\pi\times\pi}\lra X\times X $ of the quotient bundle $G\overset{\pi}\lra X=G/H$. This is a principal bundle with the  structure group $H\times H$.
\begin{lemma}\label{L:1.1}
Let $G$ be a compact Lie group, $H\subset G$ a closed
subgroup and $G\overset{\pi}\lra X=G/H$ the corresponding quotient
bundle. Then the bundle of shifts $Q\overset{\alpha}\lra X\times X$ is
a fiber bundle associated to $G\times G\overset{\pi\times \pi}\lra X\times X$.
\end{lemma}

\begin{proof} Recall that given a principal bundle $P$ over $X$ and a space $F$ 
where the structure group $T$ acts on the left by $\mu$, one can form a set of equivalence classes
\begin{equation}\label{Borel}
P\times_\mu F:=\{[p,f]\in P\times F| (p,f)\sim (pt,\mu(t^{-1})f)\}.
\end{equation}
Then $[p,f]\mapsto\pi(p)$ is a bundle projection that turns
$P\times_\mu F$ into a fiber bundle over $X$ associated to $P$ by $\mu$ \cite{Hus,MS}.

We will construct an explicit isomorphism between $Q$ and the following associated bundle. 
The group $T:=H\times H$ acts on $F:=H$ on the left by 
$$
\begin{aligned}
(H\times H)\times H &\overset{\mu}\lra H\\
((\lambda_1,\lambda_2),h) & \longmapsto\lambda_2h\lambda_1^{-1}
\end{aligned}
$$
Set $E_1:=((G\times G)\times_\mu H\overset{\pi}\lra X)$, $E_2:=Q$ and consider the
following map 
$$
\begin{aligned}
E_1 &\overset{\mathcal{F}}\lra E_2\\
[g_1,g_2,h] & \longmapsto(g_1H,g_2hg_1^{-1})
\end{aligned}
$$
To begin with $\mathcal{F}$ is well defined:
$$
g_1\lambda_1H,g_2\lambda_2,\lambda_2^{-1}h\lambda_1)\longmapsto
(g_1,\lambda_1H, g_2hg_1^{-1})=(g_1H,g_2hg_1^{-1}).
$$ 
The inverse is given by $(x,g)\overset{\mathcal{F}^{-1}}\longmapsto
[g_1,gg_1,1]$, where $g_1H=x$. If $g_1\lambda$ is chosen instead
with $\lambda\in H$ then $[g_1\lambda ,gg_1\lambda
,\lambda^{-1}1\lambda]=[g_1,gg_1,1]$ so $\mathcal{F}^{-1}$ is
well-defined. It is easy to see that it is indeed the inverse to
$\mathcal{F}$.

We claim that both diagrams commute
\bee\label{iso}
\begin{diagram}
E_1&& \rTo{\mathcal{F} }        &&E_2, &   E_2 && \rTo{\mathcal{F}^{-1}}   && E_1   \\
         & \rdTo_{\pi}   &  &\ldTo_{\alpha}&   & & \rdTo_{\alpha}   &                &\ldTo_{\pi}&\\
   &   &        X      & &   &  &                 &         X       &             &     
\end{diagram}
\eee
For instance, 
$$
(\alpha\circ\mathcal{F})([g_1,g_2,h])=\alpha
(g_1H,g_2Hg_1^{-1})=(g_1H,g_2hH)=(g_1H,g_2H)=\pi([g_1,g_2,h]).
$$
Therefore the bundle of shifts $Q=E_2$ is indeed a fiber bundle and $\mathcal{F}$ is a bundle isomorphism.
\end{proof}

For obstruction theory we follow terminology and notation of Steenrod \cite{St}. We say that 
$n$ is the {\it lowest homotopy non-trivial dimension} of $F$ if $\pi_k(F)=0$ for $1\leq k\leq n-1$ but $\pi_n (F)\neq 0$. Assume that in a fiber bundle $F\overset{\i} \hra E\overset{\pi}\lra B$
the base $B$ is a $CW$--complex and the fiber $F$ is {\it homotopy simple} up to this dimension,
i.e. $\pi_1(F)$ acts trivially on $\pi_k(F)$ for $1\leq k\leq n$.
This means that there is no obstruction to constructing a section up to
dimension $n$ and we may assume that $B^{(n)}\overset{\sigma}\lra E$
is already constructed, here $B^{(n)}$ is the $n$-skeleton of $B$.
Let $\Delta \subset B$ be an $(n+1)$ cell of $B$ which we may assume
to be contractible or even a simplex. Choosing a local trivialization we get a map $f_{\sigma}:\partial\Delta\lra F$
that defines an element of $\pi_n (F)$. It turns out that this element does not depend on a choice of trivialization and
$c_\sigma(\Delta):=[f_{\sigma}]\in\pi_n(F)$ is a $\pi_n (F)$-valued cochain
and in fact a cocycle. Its cohomology class
$\overline{c}_\sigma\in H^{n+1}(B,\pi_n (F))$ is called the primary obstruction
to extending $\sigma$. This cohomology class does not even depend on a choice
of $\sigma$ on the $n$-skeleton of $B$ and is an invariant of the bundle
$E\overset{\pi}\lra B$ itself. 
\begin{definition}[Primary characteristic class]\label{D:primchar}
The invariant $\varkappa (E):=\overline{c}_\sigma$ is called the primary characteristic class
of $E$. 
\end{definition}
The characteristic class is natural with respect to the pullback of bundles: 
$$
\varkappa (\pfi^*E)=\pfi^*\varkappa (E)
$$ 
and the Eilenberg extension theorem claims that a section $\sigma$ can be altered
on $B^{(n)}$ so as to be extendable to $B^{(n+1)}$ if and only if $\overline{c}_\sigma=0$.
This completely solves the sectioning problem when $\pi_k(F)=0$ for $n+1\leq k<\dim B$, i.e. there are no further obstructions: a section exists if and only if $\varkappa(E)=0$.

In our case the bundle in question is $H\overset{\i}\hra Q_{\pfi
,\psi}\overset{\pi}\lra M$. The fiber is a Lie group so it is
homotopy simple in all dimensions. The first non-trivial dimension
is $n=1$ as $H$ is connected and $\varkappa (Q_{\pfi
,\psi})\in H^2(M,\pi_1 (H))$. Since $\pi_2(H)=0$ for all finite-dimensional Lie groups
and $\dim M=3$ there are no further obstructions and a section exists
if and only if $\varkappa (Q_{\pfi ,\psi})=0$. Thus, we want to
compute this characteristic class. By naturality $\varkappa (Q_{\pfi
,\psi})=\varkappa((\pfi ,\psi)^*Q)=(\pfi ,\psi)^*\varkappa(Q)$ and we
need to compute $\varkappa$ for the bundle of shifts.

By \refL{1.1} $Q$ is isomorphic to the associated bundle
$\widehat{E}:=\widehat{P}\times_{\widehat{\mu}}H$ with $\widehat{P}=G\times G$
and the action
$$
\begin{aligned}
(H\times H)\times H&\ovs{\widehat{\mu}}\lra H\\
((\lambda_1 ,\lambda_2),h)&\longmapsto \lambda_2 h\lambda_1^{-1}
\end{aligned}
$$
The form of the action suggests that we can 'decompose' $\widehat{E}$ into a
combination of two simpler bundles $E$ and $E'$, namely 
\be
E:=P\times_\mu H\quad \text{with}\quad \mu(\lambda )h:=\lambda h
\ee 
and its dual 
\be 
E':=P\times_{\mu '}H\quad \text{with}\quad \mu'(\lambda)h:=h\lambda^{-1}
\ee 
(in our case $P=G$ and one can multiply on both sides). We will not explain precisely what 
the decomposition means in this case but it should be clear from the proof of \refL{1.2}(ii). Note that $E$ is
bundle isomorphic to $P$ itself by $p\mapsto[p,1]$ so we write $\varkappa(P)$ for $\varkappa(E)$.

\begin{lemma}\label{L:1.2}
Let $P\ovs{\pi}\lra X$ be a principal bundle with
the structure group $H$. Define $\widehat{P}:=(P\times P\lra X\times
X)$, $E$, $E'$, $\widehat{E}$ as above and let $\pi_1$, $\pi_2$
denote the projections from $X\times X$ to the first and the second
components. Then

{\rm(i)} $\varkappa (P)=\varkappa (E)=-\varkappa (E')$.

{\rm (ii)} $\varkappa (\widehat{E})=\pi_2^*\varkappa (P)-\pi_1^*\varkappa (P)$, if in addition $H^k(X,\mathbb{Z})=0$ for $0\leq k\leq n$, where $n$ is the lowest homotopy non-trivial dimension of $H$.
\end{lemma}
\begin{proof}
{\rm(i)} Note that if $\sigma (x)=[p,h]$ gives a section of $E$
then $\sigma '(x)=[p,h^{-1}]$ gives a section of $E'$. Also if
$\Delta\ovs{S_\Delta}\lra P|_\Delta$ is a local section of $P$ then
$$
\bal
\Delta\times F & \ovs{\Phi|_\Delta}\lra(P\times_\mu F)|_\Delta\\
(x,f) & \longmapsto [S_\Delta (x),f]\\
(\pi (p),\mu(\lambda^{-1})f) & \longmapsfrom [p,f],\quad \text{with}\ S_\Delta(\pi(p))=p\lambda,
\eal 
$$
is a local trivialization of the associated bundle.

We choose a section $S_\Delta$ of $P$ and denote $\Phi_\Delta$,
$\Phi^\prime_\Delta$ the corresponding trivializations of $E$, $E'$.
Also if $\sigma$ is the chosen section of $E$ on $B^{(n)}$ then the
$\sigma'$ is the one we choose for $E'$. By definition,
\begin{align*}
f_{\sigma'}(x)=\pi_2\circ\Phi_\Delta^{-1}\circ\sigma'(x) &=\pi_2\circ\Phi_\Delta^{-1}([p,h^{-1}]), \quad\qquad\pi (p)=x\\
&=(\pi(p),\mu'(\lambda^{-1})h^{-1}), \quad\qquad S_\Delta (\pi (p))=S_\Delta(x)=p\lambda\\
&=h^{-1}(\lambda^{-1})^{-1}=(\lambda ^{-1}h)^{-1}=(\mu(\lambda ^{-1})h)^{-1} \\
&=(\pi_2 \circ\Phi_\Delta^{-1}([p,h])^{-1})=(\pi_2\circ\Phi_\Delta^{-1}\circ\sigma(x))^{-1}=f_{\sigma}(x)^{-1}. 
\end{align*}
In other words, $c_{\sigma'}(\Delta)=[f_{\sigma}^{-1}]$. But in $\pi_n(H)$ one has $[o^{-1}]=-[o]$ \cite{Dy} for any $o$ and $\varkappa(E')=\overline{c}_{\sigma'}=-\overline{c}_\sigma =-\varkappa(E)$.

{\rm(ii)}  Under our assumptions the K\"unneth formula and the universal coefficients theorem imply that
\begin{align*}
H^{n+1}(X\times X,\pi_n(H)) &\simeq H^{n+1}(X, \pi_n (H))\oplus H^{n+1}(X,\pi_n(H)),\\
\omega &\longmapsto (\i_1^*\omega ,\i_2^*\omega)\\
\pi_1^*\omega^*+\pi_2^*\omega_2 &\longmapsfrom (\omega_1,\omega_2),
\end{align*}
where $x\ovs{\i_1} \longmapsto (x,x_0)$, $x\ovs{\i_2}\longmapsto (x_0,x)$
for some fixed point $x_0\in X$. Let $p_0\in P$ be any point with $\pi (p_0)=x_0$, then
\begin{align*} 
\i^*_1\widehat{E} &=\{(x,[p,p_0,h])\in X\times \widehat{E}|\ (x,x_0)=(\pi(p),\pi (p_0))\}\\ 
&\simeq \{(x,[p,h])\in X\times E|\ \pi (p)=x\}\simeq E'
\end{align*} 
since $p_0$ is fixed and $\widehat{\mu}$ reduces to $\mu '$ on the
first component. Analogously, $\i^*_2\widehat{E}\simeq E$. Therefore
from naturality and {\rm(i)}
\begin{align*} 
\varkappa
(\widehat{E})&=\pi_1^*\i_1^*\varkappa (\widehat{E})+ \pi_2^*\i_2^*\varkappa
(\widehat{E})=\pi_1^*\varkappa (\i_1^*\widehat{E}) +\pi_2^*\varkappa
(\i_2^*\widehat{E})\\
&=\pi_1^*\varkappa (E')+\pi_2^*\varkappa (E)
=\pi_2^*\varkappa (P)-\pi_1^*\varkappa (P)
\end{align*} 
\end{proof}
Now we are ready for the main result of this section.
\begin{theorem}\label{T:1.0}
Let $X=G/H$ be a simply connected homogeneous space, $M$ be a $3$-dimensional $CW$--complex 
and $M \xra{\psi,\pfi}X$ be continuous maps. Then a continuous $M\ovs{u}\lra G$
with $\psi=u\pfi$ exists if and only if 
$$
\psi^*\varkappa(G)=\pfi^*\varkappa (G),
$$
where $\varkappa (G)$ is the primary characteristic class of the quotient bundle $G\to X$.
\end{theorem}

\begin{proof}
In our case $P$ is the quotient bundle $G\lra X$ and we write
$\varkappa (G)$ for its primary characteristic class. It is easy to compute
$\varkappa (Q_{\pfi ,\psi})$ now since $Q_{\pfi ,\psi}=(\pfi
,\psi)^*Q$ and $Q=\widehat{E}$ for the quotient bundle $G\lra X$:
\begin{align*}
\varkappa (Q_{\pfi ,\psi}) &=\varkappa ((\pfi ,\psi)^*Q)=(\pfi ,\psi)^*\varkappa (Q)) && \text{by  naturality}\\
&=(\pfi ,\psi)^*(\pi_2^*\varkappa (G)-\pi_1^*\varkappa (G)) && \text{by \refL{1.2}}\\
&=(\pi_2\circ (\pfi ,\psi))^*\varkappa (G)-(\pi_1\circ (\pfi,\psi))^*\varkappa (G) &&\\
&=\psi^*\varkappa (G)-\pfi^*\varkappa(G).
\end{align*}
\end{proof}
In fact the conditions of \refL{1.2} are satisfied with
$n=1$ if $H$ is connected and $X$ is simply connected (simple
connectedness of $G$ is not necessary). Hence \refT{1.0} can be
applied directly to $U_n$ homogeneous spaces without reducing them
to $SU_n$ ones as long as the subgroup $H\subset U_n$ is already connected.

\section{Characterization of $2$-homotopy classes}\label{S2}

In the previous section we reduced the lifting problem to equality of pullbacks of the primary characteristic class. By obstruction theory the primary obstruction to homotopy is described in the same fashion. It turns out that this is not a coincidence and the primary characteristic class of $G\to G/H$ is essentially the same as the basic class of $G/H$. As a consequence, a lift $u$ in $\psi=u\pfi$ exists iff $\psi$ and $\pfi$ are $2$-homotopic (\refT{1.1}). 

As before we follow terminology and notation of Steenrod \cite{St}. Let $B$ be a $CW$--complex and
$B\ovs{\psi,\pfi}\lra F$ be two maps homotopic on $B^{(n-1)}$ by $\Phi$. If $\Delta\subset B^{(n)}$ is an $n$-cell
then $\partial(\Delta\times I)\simeq S^n$ and we can set
\be 
d_\Phi (\pfi ,\psi)(\Delta):=[\Phi(\partial(\Delta\times
I))]\in\pi_n (F) 
\ee 
This defines a $\pi_n (F)$--valued cochain on $B$ called {\it the difference cochain}. It turns out to
be a cocycle and its cohomology class 
$$
\overline{d}(\pfi,\psi):=\overline{d_\Phi (\pfi ,\psi)}
$$ 
does not depend on a choice of homotopy on $B^{(n-1)}$. Obviously $\overline{d}(\pfi,\psi)\in H^n (B,\pi_n (F))$. 
The homotopy $\Phi$ can be extended from $B^{(n-2)}$ to $B^{(n)}$ (it may have to be altered on
$B^{(n-1)}$) if and only if $\overline{d}(\pfi ,\psi)=0$. The difference
is natural $\overline{d} (\pfi\circ f,\psi\circ f)=f^*\overline{d}(\pfi,\psi)$ and additive
$\overline{d}(\pfi ,\chi)=\overline{d}(\pfi,\psi)+\overline{d}(\psi ,\chi)$.
Since $\pfi$ is always homotopic to itself $\overline{d} (\pfi ,\pfi)=0$ and additivity implies 
$\overline{d}(\psi,\pfi)=-\overline{d}(\pfi ,\psi)$. 

Now let $n$ be the lowest homotopy non-trivial dimension of $F$ and $F$ be homotopy simple up to this
dimension. Then any two maps into $F$ are homotopic on $B^{(n-1)}$
and ${\overline{d}(\pfi ,\psi)}$ is defined for any pair. It is called
{\it the primary difference} between $\pfi$ and $\psi$ \cite{St}. 
\begin{theorem*}[Eilenberg classification theorem] 
If the primary difference is the only obstruction to homotopy i.e. 
$\pi_k(F)=0\quad \text{for}\quad n+1\leq k\leq\dim B$,
then $\pfi,\psi$ are homotopic if and only if\ $\overline{d}(\pfi ,\psi)=0$. 
Moreover, for any $\omega\in H^n (B,\pi_n
(F))$ and a given $B\ovs{\pfi}\lra F$ there is $B\ovs{\psi}\lra F$
such that\ $\overline{d}(\pfi,\psi)=\omega$. 
\end{theorem*}
In other words, under the conditions of the theorem, maps are classified up to homotopy 
by their primary differences with a fixed map $\pfi$, and there is a one-to-one correspondence 
between homotopy classes and $H^n (B,\pi_n(F))$. In our
case $B=M$, $F=X$, $n=2$ since $X$ is simply connected and $q=1$
since generally speaking $\pi_3 (X)\neq 0$. So $M\ovs{\psi,\pfi}\lra X$ are $2$-homotopic 
if and only if $\overline{d}(\pfi,\psi)=0$. We will reexpress this condition first in terms of the basic class 
and then of the primary characteristic class.

For any connected space $F$ there are two
special self-maps, the identity $\id_F$ and the constant map
$\mbox{pt}_F (x)=x_0\in F$. The primary difference $\overline{d}(\id_F,
pt_F)$ only depends on $F$ itself since all constant maps into a connected space are homotopic to each other. 
Note that $\overline{d}(\mbox{id}_F,\mbox{pt}_F)\in H^n(F,\pi_n(F))$ and one can show \cite{St} that
\be
\overline{d}(\id_F,\pt_F)=\b_F
\ee
Now let $M\ovs{\psi, \pfi}\lra X$ be any continuous maps and
$M\ovs{\pt_{M,X}}\lra X$ be a constant map. Then by naturality and
additivity
\begin{equation}\label{e0.25}
\overline{d}(\pfi,\psi)=\pfi^*\overline{d}(\id_X,\pt_X)-\psi^*\overline{d}(\id_X,\pt_X)=\pfi^*\b_X-\psi^*\b_X.
\end{equation}
In general, $\varkappa (G)\in H^2(X,\pi_1 (H))$  and $\b_X\in H^2(X,\pi_2(X))$ but from \eqref{e0.22} we have $\pi_1(H)\simeq\pi_2(X)$ under the connecting homomorphism $\partial$. This suggests that $\varkappa (G)=\pm\partial\circ\b_X$. To prove the equality we need to use the transgression \cite{MT,St}.
\begin{definition}[Transgression]
Let $F\ovs{\i}\hra E\ovs{\pi}\lra B$ be a fiber bundle and
$\mathbb{A}$ an Abelian group. An element $\alpha\in
H^n(F,\mathbb{A})$ is called transgressive if there are cochains
$\xi\in C^n(E,\mathbb{A})$ and $\eta\in C^{n+1}(B,\mathbb{A})$ such
that $\overline{\i^*\xi}=\alpha$ and $\delta\xi=\pi^*\eta$,
where the bar denotes the corresponding cohomology
class and $\delta$ is the cohomology differential. When $\alpha$ is
transgressive, classes $\tau ^{\#}\alpha :=\overline{\eta}\in
H^{n+1}(B,\mathbb{A})$ are called its (cohomology) transgressions.
\end{definition}
\noindent Our definition follows Steenrod, but the reader is cautioned that there is another tradition in differential geometry, where the transgression goes the opposite way. There is also an analogous notion of transgression $\tau_{\#}$ in homology and the two are dual to each other, i.e. when $\alpha$ and $a$ are transgressive $\tau^\#\alpha(a)=\alpha (\tau_\# a)$. Unlike the connecting homomorphism $\partial$ which is everywhere
defined and unambiguous, both transgressions $\tau ^{\#},\tau_\#$ in general
map from a subgroup to a quotient of the corresponding (co)homology groups. The homology transgression in a sense  imitates the non-existent connecting homomorphism in homology. In particular,
spherical classes in $H_{n+1}(B,\mathbb{Z})$ are always
transgressive and the diagram
\bee\label{e0.29}
\begin{diagram}
\pi_{n+1}(B)           & \rTo{\partial}   &\pi_n(F)            \\
\dTo{\mathcal{H}_B}      &                  &\dTo{\mathcal{H}_F} \\
H_{n+1}(B,\mathbb{Z})  & \rTo{\tau_\#}     &H_n(F,\mathbb{Z})
\end{diagram}
\eee
commutes. Here $\mathcal{H}_B$, $\mathcal{H}_F$ are Hurewicz
homomorphisms and it is understood that $\mathcal{H}_F(\partial
(z))$ is just one of the transgressions of $\mathcal{H}_B(z)$.
Commutativity can be established by inspecting the definitions of
$\tau_\#$ and $\partial$ (see \cite{Hu}).

There is a case when the transgression is unambiguous. When $H^i
(B,\mathbb{A})=0$ for $0<i <k$ and $H^j(F,\mathbb{A})=0$ for $0<j<l$
a result of J.-P. Serre says that $H^m(F,\mathbb{A})\ovs{\tau^\#}\lra
H^{m+1}(B,\mathbb{A})$ is well-defined and one has the {\it Serre exact
sequence} \cite{MT}:
\bee\label{e0.30}
0\lra H^1(B,\mathbb{A})\ovs{\pi^*}\lra
H^1(E,\mathbb{A})\ovs{\i^*}\lra H^1(F,\mathbb{A})\ovs{\tau^\#}\lra
H^2(B,\mathbb{A})\ovs{\pi^*}\lra ...\ovs{\i^*}\lra
H^{k+l-1}(F,\mathbb{A}).
\eee
An analogous statement is also true for the
homology transgression. Conditions of the Serre exact sequence are
satisfied in particular if $n$, $n+1$ are the lowest homotopy
non-trivial dimensions for $F$ and $B$ respectively and $k=n+1$,
$l=n$. Under these assumtions the primary characteristic class is related straightforwardly to the basic class of the base.
\begin{lemma}\label{L:1.3}
Let $F\ovs{\i}\hra E\ovs{\pi}\lra B$ be a fiber bundle with the fiber $F$ being
homotopy simple up to dimension $n$  and let $n$, $n+1$ be the
lowest homotopy non-trivial dimensions of $F$ and $B$ respectively. Then
\bee\label{e0.32}
\varkappa (E)=-\partial\circ \b_B,
\eee
where $\pi_{n+1}(B)\ovs{\partial}\lra\pi_n (F)$ is the connecting
homomorphism (cf. \cite{Nk2}).
\end{lemma}

\begin{proof}
By the universal coefficients theorem:
$$
0\lra \Ext(H_n(B,\mathbb{Z}),\pi_n(F))\lra H^{n+1}(B,\pi_n(F))\lra
\Hom (H_{n+1}(B,\mathbb{Z}),\pi_n(F))\lra 0
$$
is exact and since $n+1$ is the lowest homotopy non-trivial
dimension of $B$ the group $H_n(B,\mathbb{Z})=0$ and the $\Ext$
term vanishes. Hence the elements of $H^{n+1}(B,\pi_n(F))$ are
completely determined by their pairing with integral homology
classes. By the Serre exact sequences both transgressions
$H^n(F,\pi_n(F))\ovs{\tau^\#}\lra H^{n+1}(B,\pi_n(F))$ and
$H_{n+1}(B,\mathbb{Z})\ovs{\tau_\#}\lra H_n(F,\mathbb{Z})$ are
unambiguous and the Whitehead transgression theorem \cite{St} (see also \cite{BH}, Appendix 1) gives $\varkappa
(E)=-\tau^\# \b_F$. Using also the duality of transgressions and \eqref{e0.29} 
\be
\bal \varkappa
(E)(a)&=-\tau^\# \b_F(a)=-\b_F(\tau_\# a)=-\mathcal{H}_F^{-1}(\tau_\#
a)=-\partial
(\mathcal{H}_B^{-1}(a))\\
&=-\partial (\b_B(a))=-\partial\circ \b_B (a).\eal
\ee
Since $a\in
H_{n+1}(B,\mathbb{Z})$ is arbitrary (all elements are spherical by the Hurewicz theorem and
hence transgressive) we get \eqref{e0.32}.
\end{proof}
Recall that the basic class regulates $2$-homotopy and the primary characteristic class 
regulates existence of a lift between two maps into $G/H$. We now establish the desired equivalence.
\begin{theorem}\label{T:1.1}
Let $X=G/H$ be a compact simply connected homogeneous space and $M$ a
$3$-dimensional $CW$--complex. Then three conditions are equivalent
for continuous $M\ovs{\psi ,\pfi} \lra X$:

{\rm(i)}\ $\pfi$, $\psi$ are $2$-homotopic, i.e. homotopic on the
$2$-skeleton of $M$;

{\rm(ii)}\ $\psi^*\b_X=\pfi^*\b_X\in H^2(M,\pi_2(X))$, where $\b_X$ is the
basic class of $X$;

{\rm (iii)}\ There exists a continuous $M\ovs{u}\lra G$ such that
$\psi =u\pfi$, where $u\pfi$ refers to the action of $G$ on $X$.
\end{theorem}

\begin{proof}
Equivalence of the first two conditions is just a particular case of the Eilenberg classification theorem. To prove {\rm (iii)} we apply \refL{1.3} to the bundle $H\hra G\lra X=G/H$ with $n=1$ since $H$ is connected and get $\varkappa(G)=-\partial\circ \b_X\in H^2 (X,\pi_1(H))$. Since $\pi_2(G)=0$ for any finite-dimensional Lie group we have from the homotopy exact sequence
\bee\label{e0.22}
0=\pi_1(G)\lla\pi_1(H)\overset{\partial}\lla\pi_2(G/H)\lla\pi_2(G)=0
\eee
that the connecting homomorphism is an isomorphism. Since it also commutes with pullbacks $\psi^*\b_X=\pfi^*\b_X$ if and only if $\psi^*\varkappa(G)=\pfi^*\varkappa (G)$. Application of \refT{1.0} now concludes the proof.

\end{proof}

\section{Secondary invariant and homotopy classes}\label{S3}

By the Eilenberg classification theorem,  maps $M\to X$ are $2$--homotopic if and only if they have the same pullbacks of the basic class $\b_X$. This pullback $\pfi^*\b_X$ is the primary invariant of a map $\pfi$. If $\pi_3(X)=0$ then the $2$-homotopy class is already the homotopy class (recall that we only consider a $3$--dimensional $M$). Otherwise, secondary invariants have to be specified. Unlike the primary invariants, classical secondary invariants are not defined constructively \cite{MT}. From \refT{1.1} we know that even $2$-homotopy of $\psi$ and $\pfi$ implies that $\psi=u\pfi$. In this section we derive an explicit characterization for such $u$ in terms of $u^*\b_G$, where $\b_G$ is the basic class of $G$. In other words, we are using $u^*\b_G$ as a secondary invariant of a pair $\psi,\pfi$ while for the lift $u$ it is a primary invariant and is defined straightforwardly.

We start with a simple observation that follows directly from the homotopy lifting property in the bundle of shifts.
\begin{lemma}\label{L:1.4}
Let $G$ be a compact connected Lie group, $H\subset G$ a closed subgroup,
$X=G/H$ and $M$ a $CW$--complex. Then two continuous maps
$M\xra{\pfi ,\psi}X$ are homotopic if and only
if there exists a nullhomotopic $M\overset{u_0} \lra G$ such that
$\psi =u_0\pfi$. Given an arbitrary map $M\overset{u}\lra G$ maps
$\pfi$, $u\pfi$ are homotopic if and only if $u=u_0w$, where $u_0$
is nullhomotopic and $w\pfi =\pfi$.
\end{lemma}

\begin{proof} Let $M\xra{1}G$ denote the constant map that maps every point into the identity of $G$. 
If $u_0^t$ is a homotopy that translates $u_0$ into $1$ then $\psi_t:=u_0^t\pfi$ translates $u_0\pfi$ into $\pfi$ and $\Phi(m,t):=(\pfi (m),\psi_t (m))$ translates $(\pfi,\pfi)$ into
$(\pfi,\psi)$. The former admits a lift $(\pfi,1)$ into $Q$, indeed $\alpha\circ(\pfi,1)=(\pfi ,\pfi)$. 
Since $Q$ is a fiber bundle by \refL{1.1} the homotopy lifting property implies that 
the following diagram can be completed as indicated: 
\be
\begin{diagram}
M\times\{0\}&  \rTo{(\pfi,1)}  & X\times G\\
\dInto    &  \ruDotsto^{\widetilde{\Phi}} & \dTo{\alpha}\\
M\times I &  \rTo^{\Phi} & X\times X
\end{diagram}
\ee
By the upper triangle $\widetilde\Phi_2 (m,0)=1$ and by the lower one
$\widetilde\Phi_1 (m,t)=\Phi_1 (m,t)=\pfi (m)$, $\widetilde\Phi_2
(m,t)\widetilde\Phi_1 (m,t)=\widetilde\Phi_2(m,t)\pfi (m)=\psi_t (m)$. Set
$u_0(m):=\widetilde\Phi_2 (m,1)$ then $u_0\pfi =\psi$ and $\widetilde\Phi_2
(\cdot ,t)$ is a homotopy that translates the constant map $1$ into
$u_0$ as required.

For the second claim note that $u=u_0w$ implies $u\pfi
=u_0w\pfi=u_0\pfi$ and is homotopic to $\pfi$. Conversely, if
$u\pfi$ is homotopic to $\pfi$ then by the first claim there is also
a second nullhomotopic $u_0$ such that $u_\pfi =u_0 \pfi$. It
suffices to set $w:=u_0^{-1}u$.
\end{proof}
Let $(M,G)$ denote the space of continuous maps $M\to G$ and $(M,G)\pfi$ the space of maps $M\to X$ that have the form $u\pfi$ for $u\in(M,G)$. \refL{1.4} suggets that the following maps play a special role.
\begin{definition}[Stabilizer]\label{D:stab}
Given a map $M\overset{\pfi}\lra X$ we call $\Stab_\pfi:=\{w\in (M,G)|w\pfi=\pfi\}$
the {\it stabilizer} of $\pfi$. 
\end{definition}
We have a natural inclusion $\Stab_\pfi\ovs{\i}\hra(M,G)$. Denote
\begin{align*}
[M,G] &:=\pi_0((M,G)),\\
[(M,G)\pfi] &:=\pi_0((M,G)\pfi).
\end{align*}
Note that $[M,G]$ is the set of homotopy classes of continuous maps $M\lra G$ and by \refT{1.1} $[(M,G)\pfi]$ is the set of homotopy classes of continuous maps into $X=G/H$, which are $2$--homotopic to $\pfi$. 

If $G$ is compact simply connected $\pi_1(G)=\pi_2(G)=0$ and it follows from the Eilenberg classification theorem that
\begin{align}\label{u*}
[M,G] &\simeq H^3(M,\pi_3(G))\notag\\
[u] &\longmapsto u^*\b_G
\end{align}
is a group isomorphism. Under this isomorphism the subgroup
$\i_*\pi_0(\Stab_\varphi)=\pi_0(\i(\Stab_\varphi))$ is mapped
into a subgroup of $H^3(M,\pi_3(G))$ that we denote $\O_\varphi$, i.e
\begin{equation}\label{Ofi}
\O_\pfi:=\{w^*\b_G\mid w\in\Stab_\pfi\}\subset H^3(M,\pi_3(G)).
\end{equation}
Although the definition \eqref{Ofi} uses the map $\pfi$ explicitly, we will show 
\begin{lemma}\label{L:1.5}
$\O_\pfi$ only depends on the $2$-homotopy class of $\pfi$ or
equivalently on $\pfi^*\b_X$ and not on $\pfi$ itself.
\end{lemma}
\begin{proof}
If $\psi$ is $2$--homotopic to $\pfi ,$ there is $M\ovs{u}\lra G$ with $\psi=u\pfi$ by \refT{1.1}. Hence
\be
\Stab_{\psi}=\{w| w\psi=\psi\}=\{w|wu\pfi=u\pfi\}=\{w|u^{-1}wu\in\Stab_\pfi\}=u(\Stab_\pfi)u^{-1}
\ee

Let $\pi_1,\pi_2$ be the natural projections from $G\times G$ to the first and the second factor
and $G\times G\ovs{m}\lra G$ be the multiplication map. Then it follows from the Hopf-Samelson theorem \cite{Dy,WG} that
\bee\label{HopSam1} 
m^*\b_G=\pi_1^*\b_G +\pi_2^*\b_G.
\eee 
This implies that given two maps $M\ovs{u,v}\lra G$ we have
$$
(u\cdot v)^*\b_G=(m\circ(u,v))^*\b_G=(u,v)^*(\pi_1^*\b_G+\pi_2^*\b_G)=u^*\b_G+v^*\b_G
$$
Using this formula we derive from definition \eqref{Ofi}
\begin{align*}
\O_\psi &=\{w^*\b_G| w\in\Stab_\psi\}=\{(uw'u^{-1})^*\b_G| w'\in\Stab_\pfi\}\\
&=\{u^*\b_G+ (w')^*\b_G-u^*\b_G| w'\in\Stab_\pfi\}= \O_\pfi 
\end{align*}
\end{proof}
Hence $ \O_\pfi = \O_{\pfi^*\b_X}$ and since every $\varkappa\in
H^2(M,\pi_2(X))$ is representable by a $\pfi$ one can talk about
$\O_\varkappa$. 
\begin{theorem}\label{T:1.2} Let $X=G/H$ be a compact simply connected homogeneous space and $M$ a
$3$-dimensional $CW$--complex. Two continuous maps $M\ovs{\psi,\pfi}\lra X$ are homotopic if and only if 
$\psi=u\pfi$ and $u^*\b_G\in\O_\pfi$ for some $M\xra{u}G$. Moreover
\begin{align}\label{secon2}
[(M,G)\pfi]\simeq H^3(M,\pi_3(G))/\O_\pfi &&(\simeq\text{means bijection}).
\end{align}
\end{theorem}
\begin{proof}
By definition of the isomorphism \eqref{u*} having $u^*\b_G\in\O_\pfi$ is equivalent to $[u]\in\i_*\pi_0(\Stab_\varphi)$ or $u$  homotopic to $w\in\Stab_\varphi$. But then $uw^{-1}$ is nullhomotopic and $\psi=(uw^{-1})w\pfi$ is homotopic to $\varphi$ by \refL{1.4}.

To prove the bijection consider the map
\begin{align*}
(M,G)&\ovs{\Pi}\lra (M,G)\pfi\\
u&\mapsto u\pfi.
\end{align*}
We will show that this is a fibration following an idea from \cite{AS}. By definition we need to complete the diagram as indicated for $A$ arbitrary and $I:=[0,1]$
\bee\label{e0.33}
\begin{diagram}
A\times\{0\}& \rTo{F_0}          &(M,G)  \\
\dInto     &\ruDotsto         &\dTo_{\Pi}   \\
A\times I   &\rTo{f}          &(M,G)\pfi
\end{diagram}
\eee
Set $\overline{F}_{0}(m,a):=F_0(a)(m)$ and $\overline{f}(m,a,t):=f(a,t)(m)$. 
Recall from \refL{1.1} that the bundle of shifts \eqref{shifts} is a fiber bundle and therefore
a fibration so the following diagram can be completed as indicated:
\be
\begin{diagram}
(M\times A)\times {0}&  \rTo{(\overline{F}_0,\pfi )}  & G\times X\\
\dInto       &  \ruDotsto^{\overline{\Phi}}    & \dTo_{\alpha}        \\
(M\times A)\times I      &  \rTo{(\overline{f},\pfi )}      &X\times X
\end{diagram}
\ee
Inspecting the definitions of $\overline{F}_{0}$, $\overline{f}$ one concludes that the original diagram can be completed as well using $\overline{\Phi}$. 

If $v\pfi =u\pfi$ then $w:=u^{-1}v \in\Stab_\pfi$ and the fiber of this fibration is exactly $\Stab_\pfi$.
Using the homotopy exact sequence of the fibration 
$$
\pi_0(\Stab_\pfi)\ovs{\i_*}\lra\pi_0((M,G))\ovs{\pi_*}\lra\pi_0((M,G)\pfi)\lra 0.
$$ 
one gets
\begin{align*}\label{secon1}
[(M,G)\pfi]\simeq \frac{[M,G]}{\i_*\pi_0(\mbox{Stab}_\pfi)} &&(\simeq\text{means bijection}).
\end{align*}
Under the isomorphism \eqref{u*} this becomes \eqref{secon2}.
\end{proof} 
For applications it is convenient to reinterpret the secondary invariant 
in terms of the deRham cohomology. Let us start with the group $H^3(G,\pi_3(G))$. Recall that we assume that $G$ is compact connected and simply connected. By the universal coefficients theorem the following sequence is exact:
\be
0\lra \mathop{\rm Tor}\nolimits(H^2(G,\Z),\pi_3(G))\lra
H^3(G,\pi_3(G))\lra H^3(G,\Z)\otimes\pi_3(G)\lra 0.
\ee
Since $G$ is a simply connected Lie group $H^2(G,\Z)=0$ and the torsion term vanishes so
\be
H^3(G,\pi_3(G))\simeq H^3(G,\Z)\otimes\pi_3(G).
\ee
Since $G$ is also compact it is a direct product
of simple components $G=G_1\times\dots\times G_N$ and therefore
$$
\pi_3(G)\simeq\pi_3(G_1)\oplus\dots\oplus\pi_3(G_N).
$$
The sum on the right $\simeq\Z^N$ because $\pi_3(\Gamma)\simeq\Z$ for any simple Lie group $\Gamma$ \cite{BtD}. 
Thus
\be
H^3(G,\pi_3(G))\simeq H^3(G,\Z)\otimes\Z^N
\ee
Assume additionally that $M$ is a closed connected $3$--manifold.
Both third cohomology groups $H^3(G,\Z)$, $H^3(M,\Z)$ are free Abelian, 
the first one by the Hurewicz theorem and the second by Poincare duality.  
This means that not only are elements of $H^3(G,\Z)\otimes\Z^N$ completely 
represented by integral classes in $H^3(G,\R)\otimes\R^N$ but also that their 
pullbacks are completely characterized as integral classes in $H^3(M,\R)\otimes\R^N$. But real cohomology classes from 
$H^3(G,\R)\otimes\R^N$ are represented by $\R^N$--valued differential $3$--forms by the deRham theorem \cite{GHV,MS}. 

Let $\Theta$ be a differential form that represents $\b_G$. Being $\R^N$--valued it is a collection
$\Theta=(\Theta_1,\dots,\Theta_N)$ of $N$ scalar $3$--forms and the pullback
$$
u^*\Theta:=(u^*\Theta_1,\dots,u^*\Theta_N)
$$
is defined as a vector--valued $3$--form. We can go one step further. 
Assuming $M$ is orientable $H^3(M,\Z)\simeq\Z$ and again by the universal coefficients:
$$
H^3(M,\pi_3(G))\simeq H^3(M,\Z)\otimes\pi_3(G)\simeq H^3(M,\Z)\otimes\Z^N\simeq\Z^N.
$$
The last isomorphism is given by evaluation of cohomology classes on the fundamental class of $M$  
or in terms of differential forms, by integration over $M$ \cite{GHV,MS}. Thus we get a combined isomorphism
\bee\label{deRham}
\bal
H^3(M,\pi_3(G)) &\ovs{\sim}\lra \Z^N\\
u^*\b_G &\longmapsto \int\limits_Mu^*\Theta:=(\int\limits_Mu^*\Theta_1,\dots,\int\limits_Mu^*\Theta_N).
\eal
\eee
Under this isomorphism the subgroup $\O_\pfi\subset H^3(M,\pi_3(G))$ is transformed into a subgroup of $\Z^N$ 
and we denote its image by the same symbol, explicitly
\begin{equation}\label{I-Ofi}
\O_\pfi:=\{\int\limits_Mw^*\Theta\mid w\in\Stab_\pfi\}\subset \Z^N.
\end{equation}
If $G$ is a simple group then $H^3(M,\pi_3(G))\simeq\Z$ and
we have explicitly 
$$
\Theta=c_G\tr(g^{-1}dg\wedge g^{-1}dg \wedge g^{-1}dg ),
$$
where $c_G$ are numerical coefficients computed in
\cite{AK1} for every simple group. Thus
\begin{equation}\label{e0.9}
u^*\Theta=c_G\tr(u^{-1}du\wedge u^{-1}du \wedge u^{-1}du).
\end{equation}
In general,
$$
\Theta_k=c_{G_k}\tr(\pr_{\g_k}(g^{-1}dg)\wedge\pr_{\g_k}(g^{-1}dg)\wedge\pr_{\g_k}(g^{-1}dg)),
$$
where $\g_k$ are the Lie algebras of $G_k$. \refT{1.2} can be restated as
\begin{corollary}\label{C:secinv}
In conditions of \refT{1.2} let $M$ be a closed connected $3$--manifold. Then two continuous maps 
$M\ovs{\psi,\pfi}\lra X$ are homotopic if and only if 
$\psi=u\pfi$ and $\int\limits_Mu^*\Theta\in\O_\pfi$ for some $M\xra{u}G$. 
\end{corollary}
If $M$ is not orientable then $H^3(M,\Z)=0$ and the secondary invariant is always $0$.

\par\vfill\pagebreak

{
\renewcommand{\baselinestretch}{1}

}

\end{document}